\documentclass[12pt]{amsart}
\usepackage{amsmath}
\usepackage{amsfonts}
\usepackage{amssymb}
\usepackage{amscd}
\usepackage{mathtools}
\usepackage{amsxtra}
\usepackage{amsthm}
\usepackage{graphicx}
\usepackage[abbrev,alphabetic]{amsrefs}
\usepackage[dvipsnames]{xcolor}
\usepackage{soul}
\setstcolor{red}
\usepackage{stmaryrd}
\usepackage{booktabs}
\usepackage{multirow}
\newtagform{tiny}{\tiny(}{)}
\usepackage{threeparttable}
\usepackage{aliascnt}

\usepackage{hyperref}
\usepackage[margin=1.25in]{geometry}
\usepackage{mathrsfs}

\usepackage{amsthm}
\usepackage{comment}
\usepackage[all,cmtip]{xy}
\usepackage{tikz-cd}
\usetikzlibrary{cd}

\tikzcdset{
  cells={font=\everymath\expandafter{\the\everymath\displaystyle}},
}

\usepackage{cleveref}

\makeatletter
\def\@tocline#1#2#3#4#5#6#7{\relax
  \ifnum #1>\c@tocdepth 
  \else
    \par \addpenalty\@secpenalty\addvspace{#2}%
    \begingroup \hyphenpenalty\@M
    \@ifempty{#4}{%
      \@tempdima\csname r@tocindent\number#1\endcsname\relax
    }{%
      \@tempdima#4\relax
    }%
    \parindent\z@ \leftskip#3\relax \advance\leftskip\@tempdima\relax
    \rightskip\@pnumwidth plus4em \parfillskip-\@pnumwidth
    #5\leavevmode\hskip-\@tempdima
      \ifcase #1
       \or\or \hskip 1em \or \hskip 2em \else \hskip 3em \fi%
      #6\nobreak\relax
    \hfill\hbox to\@pnumwidth{\@tocpagenum{#7}}\par
    \nobreak
    \endgroup
  \fi}
\makeatother

\renewcommand{\P}{\mathbb{P}}
\newcommand{\Z}{\mathbb{Z}}
\newcommand{\Q}{\mathbb{Q}}

\newcommand{\F}{\mathbb{F}}

\newcommand{\m}{\mathfrak{m}}

\newcommand{\Proj}{\mathrm{Proj}}

\def\var{\overline}

\DeclareMathOperator{\Spec}{Spec}

\DeclareMathOperator{\Ex}{Exc}

\DeclareMathOperator{\Ker}{Ker}

\DeclareMathOperator{\sht}{ht}

\theoremstyle{plain}
\newtheorem{theorem}{Theorem}[section]

\newtheorem{theoremA}{Theorem}

\crefname{theoremA}{Theorem}{Theorems}
\Crefname{theoremA}{Theorem}{Theorems}

\newaliascnt{lemma}{theorem}

\aliascntresetthe{lemma}
\crefname{lemma}{Lemma}{Lemmas}
\Crefname{lemma}{Lemma}{Lemmas}

\newaliascnt{proposition}{theorem}
\newtheorem{proposition}[proposition]{Proposition}
\aliascntresetthe{proposition}
\crefname{proposition}{Proposition}{Propositions}
\Crefname{proposition}{Proposition}{Propositions}

\newaliascnt{corollary}{theorem}

\aliascntresetthe{corollary}
\crefname{corollary}{Corollary}{Corollaries}
\Crefname{corollary}{Corollary}{Corollaries}

\newaliascnt{claim}{theorem}
\newtheorem{claim}[claim]{Claim}
\aliascntresetthe{claim}
\crefname{claim}{Claim}{Claims}
\Crefname{claim}{Claim}{Claims}

 \newtheorem*{claim*}{Claim}

\theoremstyle{definition}
\newaliascnt{definition}{theorem}

\aliascntresetthe{definition}
\crefname{definition}{Definition}{Definitions}
\Crefname{definition}{Definition}{Definitions}

\newaliascnt{notation}{theorem}
\newtheorem{notation}[notation]{Notation}
\aliascntresetthe{notation}
\crefname{notation}{Notation}{Notations}
\Crefname{notation}{Notation}{Notations}

\newaliascnt{example}{theorem}
\newtheorem{example}[example]{Example}
\aliascntresetthe{example}
\crefname{example}{Example}{Examples}
\Crefname{example}{Example}{Examples}

\theoremstyle{remark}
\newaliascnt{remark}{theorem}
\newtheorem{remark}[remark]{Remark}
\aliascntresetthe{remark}
\crefname{remark}{Remark}{Remarks}
\Crefname{remark}{Remark}{Remarks}

\newaliascnt{setting}{theorem}

\aliascntresetthe{setting}
\crefname{setting}{Setting}{Settings}
\Crefname{setting}{Setting}{Settings}

\newtheorem*{ackn}{Acknowledgements}

\newenvironment{claimproof}[0]
  {%
   \paragraph{\it Proof.}%
  }
  {%
    \hfill$\blacksquare$%
  }

\numberwithin{equation}{section}

\title[Non-quasi-$F$-split canonical affine fourfolds in every characteristic]{Non-quasi-$F$-split canonical affine fourfolds exist in every characteristic}

\author{Teppei Takamatsu}
\address{Department of Mathematics, Faculty of Science,
Saitama University,
255 Shimo-Okubo, Sakura-ku,
Saitama-shi, Saitama 338-8570,
Japan}
\email{teppeitakamatsu.math@gmail.com}

\author{Shou Yoshikawa}
\address{Institute of Science Tokyo, Tokyo 152-8551, Japan}
\email{yoshikawa.s.9fe9@m.isct.ac.jp}

\begin{document}

\begin{abstract}
We construct canonical $\Q$-factorial Gorenstein affine fourfolds in every positive characteristic that are not quasi-$F$-split.
\end{abstract}

\subjclass[2020]{14E30, 13A35}  

\keywords{quasi-F-split, positive characteristic}
\maketitle

\section{Introduction}
$F$-singularities play a central role in positive characteristic commutative algebra and algebraic geometry.
From the viewpoint of birational geometry, $F$-singularities are closely related to various classes of singularities arising in the minimal model program.
In particular, notions such as $F$-purity and strong $F$-regularity are regarded as positive characteristic analogues of log canonical and Kawamata log terminal (klt) singularities, respectively (cf.~\cite{Smith97}, \cite{hw02}).
In dimension two, the situation is relatively well understood.
It is known that klt surface singularities are $F$-pure in characteristic $p>5$ \cite{Hara98}.
In higher dimensions, however, the behavior becomes more subtle.
In dimension three, it is known that there exist canonical $\Q$-factorial singularities that are not $F$-pure in any characteristic \cite{CTW15a}.

As a generalization of $F$-purity, Yobuko \cite{Yobuko19} introduced the notion of quasi-$F$-splitting.
This concept retains many desirable features of $F$-purity while allowing greater flexibility, especially in geometric applications.
In fact, it has been shown that three-dimensional $\Q$-factorial klt singularities are quasi-$F$-split in characteristic $p>41$ \cite{KTTWYY2} (see also \cite{KTTWYY1}, \cite{KTTWYY3}).

In this paper, we construct $\Q$-factorial canonical Gorenstein affine fourfolds that are not quasi-$F$-split in any characteristic, as stated below.

\begin{theoremA}[\Cref{non-qfs-canonical}]\label{intro}
Let $k$ be an algebraically closed field of characteristic $p>0$.
Then there exists a $\Q$-factorial canonical Gorenstein affine fourfold $X$ over $k$
that is not quasi-$F$-split.
\end{theoremA}

\noindent
Note that such an example was known only in the case
$p=3$ (\cite{kty2}*{Example 4.29}).
The construction of the examples in \Cref{intro} is inspired by \cite{CTW15a}*{Section~5}.
In \cite{CTW15a}*{Section~5}, it is shown that if $f$ is a homogeneous polynomial defining a smooth supersingular (i.e.,\ non-$F$-split) elliptic curve, then the hypersurface $\F_p[x,y,z,t]/(f+t^p)$ is not $F$-pure.
However, when $f \in \F_p[x,y,z,w]$ defines a smooth supersingular (i.e.,\ non-quasi-$F$-split) K3 surface, an analogous hypersurface may still be quasi-$F$-split, as illustrated by the following example.

\begin{example}[\Cref{non-qfs-canonical}, \Cref{2^8-1}]
Let $k$ be a perfect field of characteristic $p=2$, let $A = k[x,y,z,w]$, and set
\[
f = x^4 + x^3y + y^3z + z^3w.
\]
Then $f$ defines a smooth supersingular K3 surface in $\mathbb{P}^3_k$.
The ring $A[t]/(f+t^m)$ is quasi-$F$-split for $m \leq 511$.
On the other hand, $A[t]/(f+t^{512})$ is not quasi-$F$-split.
\end{example}

\begin{remark}
In the forthcoming paper \cite{TY26}, we prove that for every quartic polynomial
$f\in k[x,y,z,w]$ defining a smooth supersingular K3 surface with Artin invariant $n$, the
hypersurface defined by $f+t^{p^{n}}$ is not quasi-$F$-split.
Furthermore, we provide a method to compute the Artin invariant of smooth quartic
K3 surfaces.
\end{remark}

\begin{ackn}
The authors wish to express their gratitude to Yuya Matsumoto, Hiromu Tanaka, Shunsuke Takagi, and Tatsuro Kawakami for valuable discussions.
The first author was supported by JSPS KAKENHI Grant Number JP25K17228.
The second author was supported by JSPS KAKENHI Grant Number JP24K16889.
\end{ackn}

\section{Construction}

\begin{notation}
Let $A$ be a polynomial ring over a perfect field $k$ of characteristic $p$.
For $g \in A$, write
\[
g=\sum_{j=1}^r M_j,
\]
where each $M_j$ is a monomial, and all of them have pairwise distinct multi-indices.
We define
\[
\Delta(g)
:= \sum_{\substack{0 \le \alpha_1,\ldots,\alpha_r \le p-1\\ \alpha_1+\cdots+\alpha_r=p}}
\frac{1}{p}\binom{p}{\alpha_1,\ldots,\alpha_r}
M_1^{\alpha_1}\cdots M_r^{\alpha_r} \in A.
\]
Here, the part in the formula where we multiply the multinomial coefficient by $1/p$ is computed in $\Z$.
\end{notation}


\begin{proposition}\label{const-f}
Let $A:=\var{\F}_p[x,y,z,w]$ and $\m:=(x,y,z,w)$.
\begin{enumerate}
    \item If $p \neq 2$, then there exists a homogeneous element $f \in A$ of degree $4$
    such that $f^{p-2} \in \m^{[p]}$ and $\Proj(A/f)$ is a smooth K3 surface.
    \item If $p=2$, then there exists a homogeneous element $f \in A$ of degree $4$
    such that 
    \[
    \Delta(f)^{2^{8}-1} \in \m^{[2^9]},
    \]
    $\Proj(A/f)$ is a smooth K3 surface, and $\sht(A/f) \geq 9$.
\end{enumerate}
\end{proposition}

\begin{proof}
First, assume that $p \equiv 3 \pmod{4}$ and set
\[
f:=x^4+y^4+z^4+w^4.
\]
Then $\Proj(A/f)$ is smooth and $f^{p-2} \in \m^{[p]}$ as desired.

Next, assume that $p \equiv 1 \pmod{4}$ and $p>5$.
By the proof of \cite{Jang}*{Theorem 2.4}, there exists $\lambda \in \var{\F_p}$ satisfying the following.
\begin{itemize}
\item 
The polynomial $f:=x^4+y^4+z^4+w^4+\lambda xyzw$ defines a smooth K3 surface $\Proj(A/f)$. Moreover, $\Proj (A/f)$ is a supersingular K3 surface with Artin invariant $1$.
\end{itemize}
More precisely, it is shown in \cite{Jang} that there exists $\lambda \in \overline{\F}_p$ with $\lambda^4 \neq 256$ such that $f^{p-1} \in \m^{[p]}$, and that such a $\lambda$ satisfies the condition stated above.
Note that the condition that $\Proj(A/f)$ is a supersingular K3 surface with Artin invariant $1$ is equivalent to the condition that the Hasse invariant vanishes to order $2$, by \cite{OgusHasse}*{Theorem~1} (cf.\ \cite{Bhatt-Singh}*{Corollary~4.10}) and \cite[Proposition~2.4]{ItoBrauer}.
Therefore, by \cite[Theorem 4.1]{Bhatt-Singh}, we have $f^{p-2} \in \m^{[p]}$ as desired.

Next, assume that $p=5$ and set
\begin{equation}
\label{eqn:char5K3}
    f:=x^4 + xy^3 + z^4 + zw^3.
\end{equation}
Then $\Proj(A/f)$ is smooth and $f^{p-2} \in \m^{[p]}$ as desired.
Note that, by \cite[Theorem 4.1]{Bhatt-Singh} again, the Hasse invariant vanishes to order 2.
As in the above argument, this implies that $\Proj(A/f)$ is a supersingular K3 surface with Artin invariant 1.

Finally, assume that $p=2$.
In this case, we set
\[
f:=x^4 + xy^3 + yz^3 + zw^3.
\]
Then $\Proj(A/f)$ is smooth.
We prove the following claim.
\begin{claim}\label{p=2,5}
Let $q$ be a positive integer with $q \equiv 26 \pmod{27}$.
Then there are no non-negative integers
$\gamma_0,\gamma_1,\gamma_2,\gamma_3$ satisfying
\[
4\gamma_0+\gamma_1 \leq q-1,\quad
3\gamma_1+\gamma_2 \leq q-1,\quad
3\gamma_2+\gamma_3 \leq q-1,\quad
3\gamma_3 \leq q-1,
\]
and
\[
\gamma_0+\gamma_1+\gamma_2+\gamma_3 = q-2.
\]
\end{claim}

\begin{claimproof}
Suppose, to the contrary, that there exist non-negative integers
$\gamma_0,\gamma_1,\gamma_2,\gamma_3$ satisfying the above conditions.
Since $q \equiv 26 \pmod{27}$, we may write
\[
q+1 = 27m
\]
for some positive integer $m$.
By summing the four inequalities, we obtain
\[
(4\gamma_0+\gamma_1)+(3\gamma_1+\gamma_2)
+(3\gamma_2+\gamma_3)+3\gamma_3
=4(\gamma_0+\gamma_1+\gamma_2+\gamma_3)
=4(q-2).
\]
Since each summand is at most $q-1$, it follows that
\[
4\gamma_0+\gamma_1 \geq q-5,\quad
3\gamma_1+\gamma_2 \geq q-5,\quad
3\gamma_2+\gamma_3 \geq q-5,\quad
3\gamma_3 \geq q-5.
\]
In particular, we have
\[
3\gamma_3 \in \{q-5,q-4,q-3,q-2,q-1\}.
\]
Thus,
\[
\gamma_3 = 9m-2 \quad \text{or} \quad 9m-1.
\]

First, assume that $\gamma_3=9m-2$.
Then $3\gamma_3=q-5$, and hence all the above inequalities are equalities:
\[
4\gamma_0+\gamma_1 = q-1,\quad
3\gamma_1+\gamma_2 = q-1,\quad
3\gamma_2+\gamma_3 = q-1.
\]
Solving these equations, we obtain
\[
\gamma_2=6m,\quad
\gamma_1=7m-\frac{2}{3},
\]
which contradicts the integrality of $\gamma_1$.

Next, assume that $\gamma_3=9m-1$.
Then
\[
q-5 \leq 3\gamma_2+\gamma_3 \leq q-1
\]
implies
\[
18m-5 \leq 3\gamma_2 \leq 18m-1,
\]
and hence $\gamma_2=6m-1$.
Similarly, from
\[
q-5 \leq 3\gamma_1+\gamma_2 \leq q-1
\]
we obtain
\[
21m-5 \leq 3\gamma_1 \leq 21m-1,
\]
which yields $\gamma_1=7m-1$.
Using
\[
\gamma_0+\gamma_1+\gamma_2+\gamma_3 = q-2,
\]
we obtain $\gamma_0=5m$.
However, we then have
\[
4\gamma_0+\gamma_1 = 20m+7m-1 = 27m-1 = q,
\]
which contradicts the inequality $4\gamma_0+\gamma_1 \leq q-1$.
This completes the proof.
\end{claimproof}

\noindent
Since every monomial $M$ appearing in
\[
\Delta(f)^{2^{8}-1}
\]
is a product of $(2^9-2)$ monomials appearing in $f$, we may write
\[
M=(x^4)^{\gamma_0}(x^3y)^{\gamma_1}(yz^3)^{\gamma_2}(zw^3)^{\gamma_3}
\]
for some non-negative integers $\gamma_0,\ldots,\gamma_3$ with
\[
\gamma_0+\cdots+\gamma_3=2^9-2.
\]
By
\begin{equation}\label{eq:p=2,5}
    (2,2^2,2^3,\ldots,2^9) \bmod 27 = (2,4,8,16,5,10,20,13,26) 
\end{equation}
it follows from  Claim~\ref{p=2,5} that
$M \in \m^{[2^9]}$.
Therefore,
\[
\Delta(f)^{2^{8}-1}
\in \m^{[2^9]},
\]
as desired.

Next, we prove that $\sht(A/f)\ge 9$.
We first prove the following claim.

\begin{claim}\label{claim:non-qfs-p=2,5}
Let $q$ be a positive integer with $q \not\equiv 1 \pmod{27}$.
Then there do not exist non-negative integers
$\gamma_0,\gamma_1,\gamma_2,\gamma_3$ such that
\[
4\gamma_0+\gamma_1 \leq q-1,\quad
3\gamma_1+\gamma_2 \leq q-1,\quad
3\gamma_2+\gamma_3 \leq q-1,\quad
3\gamma_3 \leq q-1,
\]
and
\[
\gamma_0+\gamma_1+\gamma_2+\gamma_3 = q-1.
\]
\end{claim}

\begin{claimproof}
Suppose, to the contrary, that there exist non-negative integers
$\gamma_0,\gamma_1,\gamma_2,\gamma_3$ satisfying the above conditions.
Choose an integer $r$ with $0\le r\le 26$ such that $q+r=27m$ for some positive integer $m$.

By adding the four inequalities, we obtain
\[
(4\gamma_0+\gamma_1)+(3\gamma_1+\gamma_2)+(3\gamma_2+\gamma_3)+(3\gamma_3)
\le 4(q-1).
\]
On the other hand, using $\gamma_0+\gamma_1+\gamma_2+\gamma_3=q-1$, the left-hand side equals
\[
4\gamma_0+4\gamma_1+4\gamma_2+4\gamma_3=4(q-1).
\]
Hence all the above inequalities must be equalities, that is,
\[
4\gamma_0+\gamma_1
=3\gamma_1+\gamma_2
=3\gamma_2+\gamma_3
=3\gamma_3
=q-1.
\]

Solving these equalities, we obtain
\[
\gamma_3=\frac{1}{3}(q-1),\quad
\gamma_2=\frac{2}{9}(q-1),\quad
\gamma_1=\frac{7}{27}(q-1),\quad
\gamma_0=\frac{5}{27}(q-1).
\]
In particular, $\gamma_0$ is an integer, so $27\mid (q-1)$.
This implies $q\equiv 1 \pmod{27}$, which contradicts the assumption.
\end{claimproof}

Since every monomial $M$ appearing in
\[
f\Delta(f)^{2^{n-1}-1}
\]
is a product of $(2^{n}-1)$ monomials appearing in $f$, it follows from Claim~\ref{claim:non-qfs-p=2,5} and \eqref{eq:p=2,5} that
\[
f\Delta(f)^{2^{n-1}-1} \in \m^{[2^n]}
\]
for $n \leq 9$.
Therefore, the result follows from \cite{KTY}*{Theorem~A}.
\end{proof}

\begin{remark} 
In the case where $p \equiv 1 \pmod{4}$ and $p>5$, we expect that $\lambda$ can be taken in $\F_p$.
For example, when $p<1000$, one can construct the desired $\lambda \in \F_p$ using a computer algebra system.
However, we cannot prove this fact in general.
\end{remark}

\begin{remark}
\label{rem:Schur}
The K3 surface defined by (\ref{eqn:char5K3}) is called Schur’s quartic surface (\cite{Schur}, cf.\ \cite{Nukui}).
When $p=5$, equation~(\ref{eqn:char5K3}) defines a smooth supersingular K3 surface with Artin invariant~$1$, as shown in the proof above.
Using this quartic surface, we can extend the main result of \cite{Jang} to characteristic 
five, as in the following theorem.
\end{remark}

\begin{theorem}[cf.\ {\cite{Jang}}]
\label{thm:Jang}
Let $k$ be an algebraically closed field of characteristic $p >2$.
Then every supersingular K3 surface over $k$ is isomorphic to a quartic surface in $\P^3$.
\end{theorem}

\begin{proof}
This follows from Remark \ref{rem:Schur} and the proof of \cite{Jang}*{Theorem 2.4}.
\end{proof}

\begin{proposition}\label{non-qfs-extension}
Let $k$ be a perfect field of characteristic $p>0$, let
\[
A:=k[x_1,\allowbreak\ldots,\allowbreak x_N],
\]
and let $f \in A$ be a homogeneous element of degree $N$.
Set $\m:=(x_1,\allowbreak\ldots,\allowbreak x_N)$, and let $n$ be a positive integer.
If $\sht(A/f) \geq n$ and 
\[
f_n:=f^{p-2}\Delta(f^{p-2})^{\frac{p^{n-1}-1}{p-1}} \in \m^{[p^n]},
\]
then $A[t]/(f+t^{l})$ is not quasi-$F$-split for every integer
$l \geq p^n$.
\end{proposition}

\begin{proof}
Set $g:=f+t^{l}$, $g_0:=1$, $g_1:=g^{p-2}$, and for an integer $r \geq 2$ define
\[
g_r:=g^{p-2}\Delta(g^{p-1})^{\frac{p^{r-1}-1}{p-1}}.
\]
Since
\[
g \equiv f \quad \text{and} \quad \Delta(g) \equiv \Delta(f) \pmod{(t^{l})},
\]
it follows that
\[
gg_r \equiv f^{p-1}\Delta(f)^{\frac{p^{r-1}-1}{p-1}} \pmod{\m^{[p^r]}}
\]
for $r \leq n$.
Since $\sht(A/f) \geq n$, the right-hand side is contained in $\m^{[p^r]}$ for $r \leq n-1$ by \cite{KTY}*{Theorem~5.8}.
Thus, we have $gg_r \in \m^{[p^r]}$ for $r \leq n-1$ and $\sht((A[t]/(g))_{\m}) \geq n$ by \cite{KTY}*{Theorem~A}.
Furthermore, we have
\[
g_n \equiv f_n \equiv 0 \pmod{(x_1,\ldots,x_N,t)^{[p^n]}}.
\]
Thus, for every $m \geq n$, we have
\[
gg_m
= g g_{m-n-1} g_n^{p^{m-n}}
\in (x_1,\ldots,x_N,t)^{[p^m]}.
\]
By \cite{KTY}*{Theorem~A}, the local ring
$(A[t]/(f+t^{l}))_\m$ is not quasi-$F$-split.
Therefore, $A[t]/(f+t^{l})$ is not quasi-$F$-split.
\end{proof}

\begin{remark}
\label{rem:htinfinity}
By the argument in the final part of \Cref{non-qfs-extension}, the conditions of \Cref{const-f} automatically imply that $\sht(A/f)=\infty$ (cf.\ \cite{KTY}*{Corollary 4.19}).
\end{remark}

\begin{theorem}[\Cref{intro}]\label{non-qfs-canonical}
Let $k$ be an algebraically closed field of characteristic $p>0$.
Then there exists a $\Q$-factorial canonical Gorenstein affine fourfold $X$ over $k$
which is not quasi-$F$-split.
\end{theorem}
\begin{proof}
Take a homogeneous polynomial $f(x,y,z,w) \in k[x,y,z,w]$ of degree $4$
as in \cref{const-f}.
For every integer $m>0$, set
\[
X_m := \Spec k[x,y,z,w,t]/(f(x,y,z,w)+t^m).
\]
Set $n:=4p^9+1$.
By the same arguments as in \cite{CTW15a}*{Lemmas~5.1 and~5.2} (cf.\ \cite{KTTWYY3}*{Proposition~7.11}), 
the origin is the unique singular point of $X_{n-4r}$ for every $0 \leq r \leq p^9$ and $X_n$ is canonical.
Moreover, there exists a sequence
\[
X_n=:Y_n \xleftarrow{f_{n-4}} Y_{n-4} \xleftarrow{f_{n-8}} Y_{n-8} \xleftarrow{f_{n-12}} \cdots \xleftarrow{f_1} Y_1
\]
such that the following conditions {\rm(1)--(3)} hold for every $m \in \{n, n-4, n-8, \ldots, 1\}$:
\begin{enumerate}
\item There exists an open cover $Y_m = X_m \cup U_m$ for some smooth fourfold $U_m$. 
In particular, $Y_1$ is smooth, and $Y_m$ has a unique singular point for $m>1$.
\item The morphism $f_m \colon Y_m \to Y_{m+4}$ is the blowup at the unique singular point of $Y_{m+4}$
(i.e., the origin of $X_{m+4}$).
\item Let $E_m:=\Ex(f_m)$. Then $E_m \simeq \Proj\,k[x,y,z,w,u]/(f(x,y,z,w))$, 
which is the projective cone over the K3 surface defined by $f(x,y,z,w)$.
In particular, $E_m$ is a projective normal rationally chain connected threefold with $\rho(E_m)=1$.
\end{enumerate}

By \cref{non-qfs-extension}, $X_n$ is not quasi-$F$-split.
It suffices to prove that $X_n$ is $\Q$-factorial.
Assume that $Y_m$ is $\Q$-factorial. 
By induction on $m$, it suffices to show that $Y_{m+4}$ is $\Q$-factorial.
Since $f_m \colon Y_m \to Y_{m+4}$ is a projective birational morphism between quasi-projective normal varieties and $\rho(E_m)=1$, 
we obtain $\rho(Y_m/Y_{m+4})=1$.
By the standard argument (as in the first paragraph of the proof of \cite{KM98}*{Corollary~3.18}), 
it is enough to show that every Cartier divisor $D$ on $Y_m$ with $D \equiv_{f_m} 0$ is $f_m$-semi-ample. 
This follows from \cite{Kleiman}*{Remark~9.5.25} and \cite{CT17}*{Theorem~1.1}, since $E_m$ is rationally chain connected. This completes the proof.
\end{proof}

\begin{example}\label{2^8-1}
Let $k$ be a perfect field of characteristic $p=2$ and $A=k[x,y,z,w]$.
We set
$f=x^4+xy^3+yz^3+zw^3$ and $g:=f+t^m$ for $m \leq 2^9-1=511$.
Then $A[t]/(g)$ is $10$-quasi-$F$-split.
To see this, we define $A$-module homomorphisms $u \colon F_*A \to A$ as in \cite{KTY}*{Section~1.1} and
\[
\theta:=u\bigl(F_*(\Delta(g)\cdot-)\bigr) \colon F_*A \to A.
\]
Fix
\[
a:=x^2yz^2w^3\,t^{2^{10}-2m-1}=x^2yz^2w^3t^{1023-2m},
\qquad
a_1:=ag,
\]
and define inductively
\[
a_{n+1}:=\theta(F_*(a_n)).
\]
Then $a_1\in\Ker(u)$.

\noindent\textbf{Case 1: $m$ odd.}
In this case the sequence forms a single chain:
\begin{align*}
a_1&=x^2yz^2w^3t^{1023-2m}(f+t^m)\in\Ker(u),\\
a_2&=x^3y^2z^2wt^{511-m}\in\Ker(u),\\
a_3&=xyz^2t^{255}\in\Ker(u),\\
a_4&=x^2zwt^{127}\in\Ker(u),\\
a_5&=x^3yt^{63}\in\Ker(u),\\
a_6&=x^3wt^{31}\in\Ker(u),\\
a_7&=x^3zt^{15}\in\Ker(u),\\
a_8&=xz^2wt^{7}\in\Ker(u),\\
a_9&=x^2z^2t^{3}\in\Ker(u),\\
a_{10}&=xyzwt\notin\m^{[2]}.
\end{align*}

\noindent\textbf{Case 2: $m\equiv 2\pmod 4$.}
Write $m=2r$.
Then
\begin{align*}
a_2&=x^3y^2z^2wt^{511-m}+xy^2z^2wt^{511-r} \in\Ker(u),\\
a_3&=x^3yz^2t^{255-r}+xyz^2t^{255}\in\Ker(u),\\
a_4&=x^2zwt^{127}+zwt^{127+r}\in\Ker(u),\\
a_5&=x^3yt^{63}+xyt^{63+r}\in\Ker(u),\\
a_6&=x^3wt^{31}+xwt^{31+r}\in\Ker(u),\\
a_7&=x^3zt^{15}+xzt^{15+r}\in\Ker(u),\\
a_8&=xz^2wt^{7}\in\Ker(u),\\
a_9&=x^2z^2t^{3}+z^2t^{3+r}\in\Ker(u),\\
a_{10}&=xyzwt \notin \m^{[2]}.
\end{align*}

\noindent\textbf{Case 3: $m\equiv 0\pmod 4$.}
Write $m=4s$. Then
\begin{align*}
a_2&=x^3y^2z^2w\,t^{511-4s}+xy^2z^2w\,t^{511-2s}\in\Ker(u),\\
a_3&=yz^2\left(x^3t^{255-2s}+x^2t^{255-s}+xt^{255}+t^{255+s}\right)\in\Ker(u),\\
a_4&=zw\left(x^3t^{127-s}+x^2t^{127}+xt^{127+s}+t^{127+2s}\right)\in\Ker(u),\\
a_5&=y\left(x^3t^{63}+x^2t^{63+s}+xt^{63+2s}+t^{63+3s}\right)\in\Ker(u),\\
a_6&=w\left(x^3t^{31}+x^2t^{31+s}+xt^{31+2s}+t^{31+3s}\right)\in\Ker(u),\\
a_7&=z\left(x^3t^{15}+x^2t^{15+s}+xt^{15+2s}+t^{15+3s}\right)\in\Ker(u),\\
a_8&=xz^2wt^{7}+z^2wt^{7+s}\in\Ker(u),\\
a_9&=x^2z^2t^{3}+z^2\,t^{3+2s}\in\Ker(u),\\
a_{10}&=xyzwt+yzwt^{1+s} \notin \m^{[2]}.
\end{align*}  
Therefore, $A[t]/(f+t^m)$ is $10$-quasi-$F$-split.

On the other hand, it is not $9$-quasi-$F$-split for $m \geq 2^8$ by the proof of \cref{const-f}~(2).
Indeed, every monomial $M$ appearing in $(f+t^m)\Delta(f+t^m)^{2^8-1}$ is a product of $2^9-1$ monomials appearing in $f+t^m$.
If $M$ is not contained in $\m^{[2^9]}$, then $M$ is a product either of $2^9-1$ monomials appearing in $f$, or of $2^9-2$ monomials appearing in $f$ together with  $t^m$.
However, by Claim~\ref{p=2,5}, every product of $2^9-2$ monomials appearing in $f$ is contained in $\m^{[2^9]}$.
\end{example}

\bibliographystyle{skalpha}
\bibliography{bibliography.bib}
\end{document}